\documentclass[A4,11pt]{article}

\usepackage[left=2.2cm,right=2.2cm,top=1.5in,bottom=1.4in]{geometry}
\usepackage{amsfonts}
\usepackage[fleqn]{amsmath}
\usepackage{dsfont}
\usepackage{graphicx}
\usepackage{color}
\usepackage{cite}
\usepackage{amssymb}
\usepackage{amsthm}

\usepackage[colorlinks,
            linkcolor=blue,
            anchorcolor=blue,
            citecolor=blue
            ]{hyperref}

\newtheorem{thm}{Theorem}

\newtheorem{lem}[thm]{Lemma}

\DeclareMathOperator{\supp}{supp}

\newcommand{\N}{\mathbb{N}}
\newcommand{\Z}{\mathbb{Z}}
\newcommand{\R}{\mathbb{R}}

\newcommand{\M}{\mathcal{M}}

\newcommand{\1}{\mathbf{1}}

\begin{document}

\makeatletter %将文献引用作为上标出现,增加括号，
\baselineskip=15pt %调整行间距

\title{Approximation by Finite Supported Functions\thanks{This work was supported by
National Natural Science Foundation of China (Grant Nos.11471043 and 11401451) and by the Beijing National Natural
Science Foundation (Grant No. 1172004).}}

\author{Bo Ling$^a$,\, Yongping Liu$^b$}

\date{ }
\maketitle \vspace{-0.4 cm}
\begin{flushleft}
\hspace{1cm}{\footnotesize $^a$School of Mathematics and Statistics, Xidian University, Xi'an, China\vspace{1.5mm}

\hspace{1cm}$^b$School of Mathematical Sciences, Beijing Normal
University, Beijing, China}
\end{flushleft}

% -----------------------------------------------------------
% -----------------------------------------------------------

% -----------------------------------------------------------

\maketitle

\begin{abstract}
We consider  approximation by functions with finite support  and characterize its approximation spaces in terms of interpolation spaces and Lorentz spaces.
\end{abstract}


{\bf{Subjclass:}} 41A25, 41A29, 42A10

\section{Introduction}

\quad Let $\M$ be the collection of all real-valued Lebesgue measurable functions which are finte a.e. on the real line $\R.$ For a nonzero function $f\in \M,$ the set $\{x\in \R: f(x)\neq 0 \}$ is called the support of $f$ and denoted by $\supp f.$ In other words, $\supp f=\{x\in \R: f(x)\neq 0 \}.$
Let the subset $\Sigma_\sigma,$  $\sigma>0$, consists of all $f\in \M$ with Lebesgue measure $\mu( \supp f )\leq \sigma.$ Notice that the set $\Sigma_\sigma$ is not linear, because a sum of two functions in $\Sigma_\sigma$ will in general lie in $\Sigma_{2\sigma}$.

In this article we shall consider  approximation   by  finite-supported functions  from $\Sigma_\sigma$ in $L_p(\R)$ space for $0<p< \infty.$
Given a function $f$, we define the approximation error by
\begin{equation}\label{error} E_\sigma(f)_p:=\inf_{g\in \Sigma_\sigma} \|f-g\|_{L_p}.\end{equation}
Note that it is not necessary to assume that $f\in L_p$ in the definition (\ref{error}).

We are interested  in describing the functions $f$ for which $E_\sigma(f)_p$
 has a prescribed asymptotic behavior as $\sigma$ increases to  $\infty$. Therefore we define the  approximation spaces, which are a collection of functions with common upper bounds for the errors of approximation. They have been studied in various contexts, for details see Chapter 7 in \cite{devore1993constructive} and references there.
For each $\alpha>0$ and $0<q\leq \infty,$ we define the approximation space $A^\alpha_{p,q}$ as the set of all $f\in \M$ such that $\|f\|_{ A^\alpha_{p,q}}$ is finite, where
\begin{equation}\label{app-space} \|f\|_{ A^\alpha_{p,q}} := \begin{cases}
\left( \displaystyle\int_0^\infty  \left[\sigma^\alpha E_\sigma(f)_p\right]^q \frac{d\sigma}{\sigma}  \right)^{\frac{1}{q}}, &0<q<\infty,\\
\sup_{\sigma>0} \sigma^\alpha E_\sigma(f)_p, &  q=\infty.
\end{cases} \end{equation}
 It can be proved that $\|\cdot\|_{A^\alpha_{p,q}}$ is a quasinorm for the space $A^\alpha_{p,q}$ and is homogeneous.  If $\|f\|_{A^\alpha_{p,q}}=0,$ then $E_\sigma(f)_p=0$ for all $\sigma>0,$ which implies that the measure of $\supp(f)$ is $0$ and hence $f=0, a.e. x\in \mathbb{R}.$ From the facts that $E_{2\sigma} (f+g)_p \leq E_{\sigma} (f)_p +E_{\sigma} (g)_p, f,g\in A^\alpha_{p,q},$  and $E_{\sigma} (\lambda f)_p=\lambda E_{\sigma} (f)_p,$ for any $\lambda\ge 0,$    we may derive that $$\|f+g\|_{A^\alpha_{p,q}}\leq 2^\alpha\left( \|f\|_{A^\alpha_{p,q}}+\|g\|_{A^\alpha_{p,q}} \right).$$
It is easy to see that the spaces $A^\alpha_{p,q}$ is decreasing  as $q$ decreasing for fixed $\alpha.$ But unlike the most cases  $A^\alpha_{p,q}$ is not decreasing  as $\alpha$ increasing here.

In this paper, we are mainly concerned with characterization of  the approximation spaces $A^\alpha_{p,q}.$ It will be found that $A^\alpha_{p,s}$  is equivalent to a Lorentz space. The analogous results in the discrete cases  were found by Devore in \cite{devore1998nonlinear}, where the n-term approximation of a $l_2$ sequence   was used to illustrate the nonlinear approximation in a Hilbert spaces.

An outline of this paper is as follows. In section 2, we recall some necessary results about non-increasing  rearrangements and Lorentz spaces which are used in characterizing approximation spaces.
In section 3, we discuss existence of best approximation elements.
In section 4, we characterize the approximation space when $s=\infty,$ that is, all functions with a common approximation order.
In Section 5, we introduce the K-functionals and discuss its relations to best approximation.
In Section 6, we characterize the approximation spaces in the general cases.

\section{Decreasing  rearrangement and  Lorentz spaces}
For a function $f\in \M,$ we define the distribution function $\mu_f(\lambda):=\mu\{x\in \R: |f(x)|>\lambda \}$ for $\lambda\geq 0.$  The function $\mu_f(\lambda)$ is nonnegative, monotone decreasing and right-continuous. A function $f\in \M$ is said to vanish at infinity if $\mu_f(\lambda)$ is finite for all $\lambda>0.$ We denote  $\M_0$ the space consisting of all functions which \emph{vanish at infinity}.
For each $f$ we define its decreasing  rearrangement $f^\ast$ by $f^\ast(t)=\inf\{\lambda: \mu_f(\lambda)\leq t\}$ for  $t\geq 0.$

It is worth mentioning the following results.
For $f\in \M$ we have $\lim_{\lambda\to \infty}\mu_f(\lambda)=0$ and $f^\ast(t)$ is finite for $t>0$ from a.e. finitness of $f.$ If, in addition, $f$   vanishes at infinity, then $\mu_f(\lambda)$ is finite for $\lambda>0$ and  $f^\ast(t)$ vanishes at infinity.

Let $0<p<\infty$ and $0<q\leq \infty.$ For a measurable function $f$ on the real line define
\begin{equation}\label{Lorentz}  \|f\|_{L_{p,q}}:=\begin{cases}  \left(\displaystyle \int_0^\infty  (t^\frac{1}{p} f^\ast (t) )^q \frac{dt}{t} \right)^{\frac{1}{q}} , & 0<q<\infty,\\
  \sup_{t>0} t^{\frac{1}{p}} f^\ast(t), & q=\infty.  \end{cases}  \end{equation}
The set of all $f$ with $\|f\|_{L_{p,q}}<\infty$ is denoted by $L_{p,q}$ and is called the Lorentz space with indices $p$ and $q.$ It is known that $L_{p,p}=L_p$ and $L_{p,\infty}$ is weak $L_p.$ For fixed $p,$ the Lorentz spaces $L_{p,q}$ increase as the exponent $q$ increases.

By the following Lemma, we represent the  approximation error $E_\sigma(f)_p$ by its decreasing rearrangement. Its direct result is the existence of best approximation in the next section.
\begin{lem}\label{lem:basic}
Suppose that $f$ belongs to $\M_0$ and let $\sigma>0$ and $0<p\leq \infty.$ Then there is a measurable set $A_\sigma,$ with $\mu(A_\sigma)=\sigma,$  such that
\begin{equation}\label{eq:equimeas} \int_{A_\sigma} |f|^p  \,d\mu =\int_0^\sigma |f^\ast(t)|^p \,dt \mbox{\quad and \quad }  \int_{A_\sigma^C} |f|^p  \,d\mu =\int_\sigma^\infty |f^\ast(t)|^p \,dt . \end{equation}
Moreover, for every set $B_\sigma$ with $\mu(B_\sigma)=\sigma,$ it holds that
\[ \int_{B_\sigma} |f|^p \,d\mu \leq   \int_{A_\sigma} |f|^p \,d\mu \mbox{ and }
\int_{B^c_\sigma} |f|^p \,d\mu \geq   \int_{A^c_\sigma} |f|^p \,d\mu,     \]
 the sets $A_\sigma$ can be constructed to increase with $\sigma,$ i.e.,
\[ A_{\sigma_1} \subset A_{\sigma_2} \mbox{ \quad for }0<\sigma_1< \sigma_2<\infty.  \]
\end{lem}
Note that  if one side of each equation in (\ref{eq:equimeas}) is infinity so is the other side.
\begin{proof}
First suppose  that $\sigma$ lies in the range of the distribution function $\mu_f$ of $f.$ That is, there exists $\alpha$ for which $\mu_f(\alpha)=\sigma.$ Form the monotone decreasing of $\mu_f$ and the definition of $f^\ast,$ it follows that
\[ f^\ast(t)=\inf\{\lambda: \mu_f(\lambda) = t  \}  \]
and then the right-continuity of $\mu_f$ gives $\mu_f(f^\ast(\sigma))=\sigma.$ It means that the set $A_\sigma:=\{x: |f(x)|>f^\ast(\sigma) \}$ has measure $\mu(A_\sigma)=\sigma$ and the distribution functions of $f \1_{A_\sigma}$ and $f \1_{A_\sigma^c}$ are
\[\mu_{f \1_{A_\sigma}}(\lambda)=\begin{cases} \mu_f(\lambda), &    \lambda>f^\ast(\sigma),
\\  \sigma,&  0\leq \lambda \leq f^\ast(\sigma),
   \end{cases} \mbox{\quad and \quad } \mu_{f \1_{A_\sigma^c}}(\lambda)=\begin{cases} 0, &    \lambda>f^\ast(\sigma),
   \\  \mu_f(\lambda)-\sigma,&  0\leq \lambda \leq f^\ast(\sigma), \end{cases} \]
respectively.

On the other hand, the distribution functions of $f^\ast \1_{[0,\sigma]}$ and $f^\ast \1_{[\sigma,\infty]}$ are
\[ \mu_{f^\ast \1_{[0,\sigma]}}=\begin{cases} \mu_{f^\ast}(\lambda), &    \lambda >f^\ast(\sigma),\\
 \sigma,&  0\leq \lambda \leq f^\ast(\sigma),    \end{cases} \mbox{\quad and \quad }
 \mu_{f^\ast \1_{[\sigma,\infty]}}=\begin{cases} 0, &    \lambda >f^\ast(\sigma),\\
 \mu_{f^\ast}(\lambda)- \sigma,&  0\leq \lambda \leq f^\ast(\sigma).    \end{cases}
  \]
Therefore, the equimeasurability of $f$ and $f^\ast$ gives  the equimeasurability of  $f \1_{A_\sigma}$ and $f^\ast \1_{[0,\sigma]}$, and the equimeasurability of $f \1_{A^c_\sigma}$ and $f^\ast \1_{[\sigma,\infty]}.$ Notice that the set $A_\sigma$ increase with $\sigma.$ Further, the equimeasurability and the layer cake  representation theorem imply the equations (\ref{eq:equimeas}).

Next we consider the case where $\sigma$ is not in the range of $\mu_f.$  Let $\lambda_0=f^\ast(\sigma).$

If $\lambda_0=0,$ we have $\mu\{x:|f(x)|>0\}=:\sigma_0<\sigma.$ In this case, we choose $A_\sigma:=\supp(f)\cup B_\sigma,$ where $B_\sigma$ has measure $\sigma-\sigma_0$ and is disjoint from $\supp(f).$ It is obvious that
\[ \int_{A_\sigma} |f|^p \,d\mu=\int_{\supp(f)} |f|^p\,d\mu=\int_0^{\sigma_0} |f^\ast(t)|^p \,dt= \int_0^{\sigma} |f^\ast(t)|^p \,dt \]
from $\sigma_0$ lying in the range of $\mu_f.$

If $\lambda_0>0,$  we have $\mu_f(\lambda_0)=:\sigma_0<\sigma\leq \sigma_1:=\mu_f(\lambda_0-).$
This shows that
\begin{equation}\label{eq:f-ast-cons} f^\ast(t)=\lambda_0,    \qquad t\in [\sigma_0,\sigma_1),  \end{equation}
from the definition of $f^\ast.$   We can prove that
\begin{equation}\label{eq:fast2} \sigma_1=\mu\{x:|f(x)|\geq \lambda_0 \}.\end{equation}
Combining with $\mu_f(\lambda_0)=\sigma^0,$ we obtain that the set $B:= \{x:|f(x)|=\lambda_0  \}$ has measure $\sigma_1-\sigma_0.$ Set $A_\sigma=\{x:|f(x)|>f^\ast(\sigma)\} \cup B_\sigma$ where $B_\sigma$ is a subset of $B$ with $\mu(B_\sigma)=\sigma-\sigma_0.$ It holds that $\mu(A_\sigma)=\sigma$ and
\[ \int_{A_\sigma} |f|^p \,d\mu=\int_{\{x:|f(x)|>f^\ast(\sigma)\}} |f|^p\,d\mu+\int_{B_\sigma} |f|^p\,d\mu  =\int_0^{\sigma_0} |f^\ast(t)|^p \,dt+[f^\ast(\sigma)]^p(\sigma-\sigma_0)= \int_0^{\sigma} |f^\ast(t)|^p \,dt  \] from $\sigma_0$ lying in the range of $\mu_f$ and (\ref{eq:f-ast-cons}).
It remains to prove (\ref{eq:fast2}). Since $\{x:|f(x)|\geq \lambda_0 \}=\cap_{n\in \N_+} \{x:|f(x)| > \lambda_0-\frac{1}{n} \}$ and $\mu_f(\lambda_0-\frac{1}{n})=\mu(\{x:|f(x)| > \lambda_0-\frac{1}{n} \})<\infty$ for $f$ vanishing at infinity, we have
\[ \mu\{x:|f(x)|\geq \lambda_0 \}=\lim_{n\to \infty} \mu_f(\lambda_0-\frac{1}{n})=\mu_f(\lambda_0-)=\sigma_1.\]

In all the cases above, we have
\[ f(x) \geq f(y),  \mbox{ for any } x\in A_\sigma,\, y\in A_\sigma^c.   \]
Hence it hold for any $B_\sigma$ with $\mu(B_\sigma)=\sigma$ that
\[ \int_{B_\sigma} |f|^p \,d\mu  = \int_{B_\sigma\cap A_\sigma} |f|^p \,d\mu +\int_{B_\sigma\cap A_\sigma^c} |f|^p \,d\mu \leq  \int_{B_\sigma\cap A_\sigma} |f|^p \,d\mu +\int_{B^c_\sigma\cap A_\sigma} |f|^p \,d\mu =   \int_{A_\sigma} |f|^p \,d\mu  \]
and
\[\int_{B_\sigma^c} |f|^p \,d\mu  = \int_{B^c_\sigma\cap A_\sigma} |f|^p \,d\mu +\int_{B^c_\sigma\cap A_\sigma^c} |f|^p \,d\mu \geq \int_{B_\sigma\cap A^c_\sigma} |f|^p \,d\mu +\int_{B^c_\sigma\cap A_\sigma^c} |f|^p \,d\mu = \int_{A_\sigma^c} |f|^p \,d\mu.\]
\end{proof}

\section{Existence of best approximation}
\begin{thm}\label{thm-exist}
Let $0<p< \infty$ and $\sigma>0.$ Then for a function $f$ in $L_p$ or $L_{q,\infty}, 0<q<p,$ there exists a best approximation $f_\sigma$ to $f$ from $\Sigma_\sigma$  in the $L_p$-norm, i.e.,
\[  E_\sigma(f)_p=\|f-f_\sigma\|_{L_p}.  \]
\end{thm}
Note that the best approximation does not lying in $L_p$  for a function $f$ in $L_{q,\infty},$ but not in $L_p$.
\begin{proof}
For any function  $g$  in $\Sigma_\sigma$ with its support  $B_\sigma,$ it holds
\[ \|f-g\|_{L_p}^p = \int_{B_\sigma} |f-g|^p \,d\mu +\int_{B_\sigma^c} |f|^p \,d\mu \geq \int_{B_\sigma^c} |f|^p \,d\mu  \geq \int_{A_\sigma^c} |f|^p \,d\mu= \int_\sigma^\infty |f^\ast(t)|^p\,dt.  \]
On the other hands, let $A_\sigma$ be given as in Lemma \ref{lem:basic},  we have
\[ \|f-f \1_{A_\sigma}\|_{L_p}^p = \int_\sigma^\infty |f^\ast(t)|^p\,dt.   \]
Therefore, $f \1_{A_\sigma}$ is a best approximation of $f$ from $\Sigma_\sigma$ in $L_p-$norm, and the error is given by
\[  E_\sigma(f)_p=\int_\sigma^\infty |f^\ast(t)|^p\,dt.  \]
For a function $f\in L_p$ the error $E_\sigma(f)_p$ is finite, while for $f\in L_{q,\infty} (0<q<p)$ it will be proved that $E_\sigma(f)_p$ is finite in Theorem \ref{thm-app1} which implies the existence of best approximation.
\end{proof}

\section{Characterization of approximation spaces when $q=\infty$ }

We characterize the approximation spaces  $A_{p,\infty}^\alpha$ in this section, i.e.,  given $0<p<\infty$ and $\alpha>0,$ for which function $f$   it holds
\[ E_\sigma(f)_p \leq C\sigma^{-\alpha},  \qquad\sigma>0, \]
for some constant $C.$
\begin{thm}\label{thm-app1}
Let $f\in \M,$  $0<p<\infty$ and $\alpha >0.$ Then
\begin{equation}{\label{Jackson}} E_\sigma(f)_p \leq C\sigma^{-\alpha}
\end{equation}
for some constant $M>0$
if and only if $f\in L_{p_1,\infty},$ where $\alpha=\frac{1}{p_1}-\frac{1}{p}.$
Moreover the infimum $C_0$ of all $C$ which satisfy (\ref{Jackson}) is equivalant to $\|f\|_{L_{p_1,\infty}}$ in the sense that
\[   c_1 \|f\|_{L_{p_1,\infty}} \leq C_0 \leq   c_2 \|f\|_{L_{p_1,\infty}}  \]
where two constants $c_1$ and $c_2$ depend  only on $p$ and $\alpha.$
\end{thm}
In other words, Theorem \ref{thm-app1} means that
\[A_{p,\infty}^\alpha=L_{p_1,\infty}   \]
 where $\alpha=\frac{1}{p_1}-\frac{1}{p}.$

\begin{proof}
If $f\in L_{p_1,\infty}$ with $p_1=\frac{p}{\alpha p+1},$  then $\|f\|_{L_{p_1,\infty}}:=\sup_{t>0} t^{\frac{1}{p_1}} f^\ast(t) <\infty,$ that is, $f^\ast(t) \leq  \|f\|_{L_{p_1,\infty}} t^{-\frac{1}{p_1}}$ for all $t>0.$
We have
\[
E_\sigma(f)_p^p=\int_{[\sigma,\infty]} f^\ast(t)^p \,dt \leq  \frac{p_1}{p-p_1} \|f\|_{L_{p_1,\infty}}^p \sigma^{-\frac{p}{p_1}+1}.
\]
Therefore, if $f\in L_{p_1,\infty},$ then clearly
\begin{equation}\label{eq1} E_\sigma(f)_p \leq  (\alpha p)^{-1/p}\|f\|_{L_{p_1,\infty}} \sigma^{-\alpha}.  \end{equation}

On the other hand, if  $E_\sigma(f)_p \leq C\sigma^{-\alpha} $ for some constant $C>0,$ then
\[ f^\ast(2\sigma)^p \leq \frac{1}{\sigma} \int_{[\sigma,2\sigma]} f^\ast(t)^p \,dt \leq \frac{1}{\sigma} E_\sigma(f)_p^p \leq C^p \sigma^{-\alpha p -1}.  \]
Therefore, we have
\[  f^\ast(\sigma) \leq 2^{\alpha +\frac{1}{p}} C \sigma^{-1/p_1}, \quad    \sigma >0. \]
It implies $f\in L_{p_1,\infty}$ and
\begin{equation}\label{eq2} \|f\|_{L_{p_1,\infty}} \leq 2^{\alpha +\frac{1}{p}} C.  \end{equation}

By inequalities (\ref{eq1}) and  (\ref{eq2}), we have
\[  2^{-\alpha-1/p} \|f\|_{L_{p_1,\infty}}  \leq  M_0\leq  (\alpha p)^{-1/p}\|f\|_{L_{p_1,\infty}}.  \]
\end{proof}

\section{K-functional and best approximation}
We start with Bernstein-type inequality of finite-supported functions, which is used to prove  the reverse part of Theorem  \ref{thm-DR}.
\begin{thm}[Bernstein's inequality]
Let $0<p_1<p<\infty$ and $\phi \in \Sigma_\sigma \cap L_p.$   Then  $\phi\in L_{p_1,\infty}$ and  it also holds
\begin{equation}{\label{Bern}} \|\phi\|_{L_{p_1,\infty}} \leq C \sigma^r \|\phi \|_{L_p} \end{equation}
where $r=\frac{1}{p_1}-\frac{1}{p}.$
\end{thm}
\begin{proof}
From $\phi \in \Sigma_\sigma$, we have
\[ \|\phi \|_{L_{p_1,\infty}} := \sup_{t>0} t^{\frac{1}{p_1}} \phi^\ast(t)= \sup_{0<t\leq \sigma} t^{\frac{1}{p_1}} \phi^\ast(t).  \]
For each $t\in (0,\sigma],$ we see that
\[ (t^{\frac{1}{p_1}} \phi^\ast(t))^p \leq t^{\frac{p}{p_1}}  \frac{\int_0^t \phi^\ast(s)^p\,ds }{t}=t^{p r } \int_0^t \phi^\ast(s)^p\,ds \leq t^{rp} \|\phi\|_{L_p}^p \]
and (\ref{Bern}) follows by taking an supremum over $t\in (0,\sigma].$
\end{proof}

For $0<p_1<p<\infty,$ the K-functional for a function $f\in L_p+L_{p_1,\infty}$ is defined by
\[K(f,t;L_p,L_{p_1,\infty}):=  \inf\{ \|f_0\|_{L_p} +t\|f_1\|_{L_{p_1,\infty}}: f=f_0+f_1  \}. \]
We obtain direct and inverse theorem characterized by the above K-functional.
\begin{thm}\label{thm-DR}
Let $0<p_1<p<\infty$ and $f\in L_p+L_{p_1,\infty}.$  Then we have
\begin{equation}\label{ba-Kf} E_\sigma(f)_p\leq  C K(f,\sigma^{-r}; L_p, L_{p_1,\infty}) \end{equation}
and on the other hand
\begin{equation}\label{inverse}
 K(f,\sigma;L_p,L_{p_1,\infty}) \leq C \sigma^{-r} \int_0^\sigma [ t^r E_t(f)_p] \frac{dt}{t},
\end{equation}
where $r=\frac{1}{p_1}-\frac{1}{p}.$
\end{thm}

\begin{proof} Let $f\in L_p+L_{p_1,\infty}.$ Then, there exist $f_0\in L_p$ and $f_1\in L_{p_1,\infty}$ such that $f=f_0+f_1.$ By Theorem \ref{thm-exist},  $f_1$ has a best approximation $g_\sigma$ from $\Sigma_\sigma$ in $L_p-$norm. Then from equation (\ref{Jackson}) we have
\[
E_\sigma(f)_p\leq \|f-g_\sigma\|_{L_p} \leq \|f_0\|_{L_p} + E_\sigma(f_1)_p \leq C( \|f_0\|_{L_p} +\sigma^{-r}\|f_1\|_{L_{p_1,\infty}})
\]
and (\ref{ba-Kf}) follows by taking an infimum over all decomposition $f=f_0+f_1.$

On the other hand,
by the monotonity of $K(f,t;L_p,L_{p_1,\infty})$ and $E_t(f)_p$ it suffices to prove (\ref{inverse}) for $\sigma=2^m, m\in \Z.$
It is easy to obtain
\[ 2^{-mr} \int_0^{2^m} [ t^r E_t(f)_p] \frac{dt}{t} \geq C_r 2^{-mr} \sum_{k=-\infty}^m 2^{kr} E_{2^k}(f)_p.   \]
 Let $\varphi_m$ be the best approximation of $f$ from $\Sigma_{2^m}$ in $L_p$-norm for each $m\in \Z.$ We have
\begin{align}
K(f, 2^{-mr};L_p,L_{p_1,\infty})&\leq \|f-\varphi_m\|_{L_p} +2^{-mr} \|\varphi\|_{L_{p_1,\infty}}   \nonumber\\
             &\leq E_{2^m}(f)_p+2^{-mr} \sum_{k=-\infty}^m \|\varphi_k-\varphi_{k-1}  \|_{L_{p_1,\infty}} \nonumber\\
             &\leq E_{2^m}(f)_p+2^{-mr} \sum_{k=-\infty}^m 2^{kr} \|\varphi_k-\varphi_{k-1}  \|_{L_p}    \label{eq11}\\
             &\leq E_{2^m}(f)_p+2^{-mr} \sum_{k=-\infty}^m 2^{kr} 2 E_{2^{k-1}}(f)_p   \label{eq21} \\
             &\leq C_r  2^{-mr} \sum_{k=-\infty}^m 2^{kr} E_{2^k}(f)_p \leq  2^{-mr} \int_0^{2^m} [ t^r E_t(f)_p] \frac{dt}{t},\nonumber
\end{align}
where (\ref{eq11}) follows from Bernstein-type inequality (\ref{Bern}) and (\ref{eq21}) from triangular inequality.
\end{proof}

\section{Characterization of Approximation Spaces}
For $0<\theta<1$ and $0<q\leq \infty,$ the interpolation space $(L_p, L_{p_1,\infty})_{\theta,q}$ is defined as the set of all functions $f\in L_p+L_{p_1,\infty}$ such that
\[|f|_{(L_p, L_{p_1,\infty})_{\theta,q}}:= \begin{cases} \left(\displaystyle\int_0^\infty [t^{-\theta} K(f,t;L_p,L_{p_1,\infty})]^q \frac{dt}{t}\right)^{1/q}, & 0< q<\infty\\  \sup_{t>0} t^{-\theta} K(f,t;L_p,L_{p_1,\infty}),   & q=\infty,  \end{cases}  \]
is finite.
We characterize completely the approximation space $A_{p,q}^\alpha$ by means of the interpolation spaces $(L_p, L_{p_1,\infty})_{\theta,q}$ in this section.
\begin{thm}\label{thm:interp} Let $0<p<\infty, 0<\alpha<r<\infty,$ and $r=\frac{1}{p_1}-\frac{1}{p}.$ Then, there holds the following equality
\[A^\alpha_{p,q}=(L_p,L_{p_1,\infty})_{\alpha/r,q} \]
with equivalent norm.
\end{thm}

Following Theorem 5.3.1 \cite{bergh2012interpolation}, which characterize the interpolation spaces between Lorentz spaces, we have the following Theorem \ref{thm:final}. That is, the approximation spaces of finite-supported functions in $L_p$-norm is Lorentz spaces.
\begin{thm}\label{thm:final}
Let $0<p<\infty, 0<q\leq \infty,$ and  $0<\alpha<\infty.$ Then, there are the approximation spaces
\[ A^\alpha_{p,q}= L_{p_1,q} \]
with equivalent norm,
where $p_1$ satisfies $\alpha=\frac{1}{p_1}-\frac{1}{p}.$
\end{thm}
\begin{proof}
By Theorem 5.3.1 \cite{bergh2012interpolation}, we have $ (L_p,L_{p_1,\infty})_{\alpha/r,q} = L_{p_2,q}, $
where
\[  \frac{1}{p_2}=\frac{1-\alpha/r}{p} +\frac{\alpha/r}{p_1}. \]
Since $r=\frac{1}{p_1}-\frac{1}{p},$  the index $p_2$ satisfies   $\alpha=\frac{1}{p_2}-\frac{1}{p}.$ Therefore we prove Theorem \ref{thm:final} from Theorem \ref{thm:interp}.
\end{proof}

It remains to prove Theorem \ref{thm:interp}, for which we need a variant of Hardy's inequality.
\begin{lem}\label{lem:Hardy}
Let $0<q<\infty$ and $\theta>0.$ Then the inequality
\begin{equation}\label{Hardy}
 \int_0^\infty \left(\sigma^{-\theta} \int_0^\sigma \phi(t) \frac{dt}{t}  \right)^q \frac{dt}{t}
\leq C(\theta,q) \int_0^\infty \left(t^{-\theta} \phi(t)\right)^q \frac{dt}{t}.
\end{equation}
is valid for $\phi(t)=t^r \varphi(t),$  where $r>0$ and $\varphi$ is  any non-negative decreasing function on $\R_+$.
\end{lem}
It was showed in $\S$3 Chapter 2\cite{devore1993constructive} that the inequality (\ref{Hardy}) holds for $1\leq q\leq \infty$ and any non-negative measurable function $\phi,$ or even for $0<q\leq \infty,$ provide the function $\phi$ is monotone. We can prove Lemma \ref{lem:Hardy} in the same way.
\begin{proof}
Let $a_k:=\phi(2^{-k})=2^{-kr} \varphi(2^{-k})$ and $b_k:=\int_0^{2^{-k+1}} t^r \varphi(t)\frac{dt}{t}$ for $k\in \Z.$ We have
\begin{equation}
b_k=\sum_{j=k-1}^\infty \int_{2^{-j-1}}^{2^{-j}} t^r \varphi(t)\frac{dt}{t} \leq 2^r \sum_{j=k-1}^\infty a_j.
\end{equation}
We can apply Lemma 3.4 in Chapter 2\cite{devore1993constructive} and obtain
\[\|(b_k)\|_{\theta,q}\leq C \|(a_k)\|_{\theta,q},   \]
where
$$\|(a_k)\|_{\theta,q}=\begin{cases} \left(\displaystyle\sum_{k\in \Z} \left(2^{k\theta} a_k \right)^q\right)^{1/q}, & 0< q <\infty,
\\ \sup_{k\in \Z} 2^{k\theta} a_k, & q=\infty. \end{cases}$$
The left side of (\ref{Hardy}) is
\[ \sum_{k\in \Z} \int_{2^{-k-1}}^{2^{-k}} \left(\sigma^{-\theta} \int_0^\sigma t\varphi(t)\frac{dt}{t} \right)^q \frac{d\sigma}{\sigma} \leq
\sum_{k\in \Z}   \left(2^{(k+1)\theta} \int_0^{2^{-k}} t\varphi(t)\frac{dt}{t} \right)^q= \|(b_k)\|_{\theta,q}^q.  \]
Similarly, the right integral in (\ref{Hardy}) is larger than $2^{-(r+1)}  \|(a_k)\|_{\theta,q}^q,$ and therefore we prove (\ref{Hardy}).
\end{proof}

\begin{proof}[Proof of Theorem \ref{thm:interp}]
Let $X:=(L_p,L_{p_1,\infty})_{\alpha/r,q}.$  For a function $f\in X,$
\begin{equation}\label{Int-norm}
\|f\|_X:= \begin{cases} \left\{\displaystyle\int_0^\infty \left[t^{-\alpha/r} K(f,t;L_p,L_{p_1,\infty})\right]^q \frac{dt}{t}\right\}^{1/q}, & 0< q<\infty,
\\  \sup_{t>0} t^{-\alpha/r} K(f,t;L_p,L_{p_1,\infty}),   & q=\infty.  \end{cases}
\end{equation}
For $f\in X$, equation (\ref{ba-Kf}) yields $\|f\|_{ A^\alpha_{p,q}}\leq C \|f\|_X$ by a change of variables.

On the other hand, for $f\in A^\alpha_{p,q},$ we have
\begin{equation}
\|f\|_X =  \left\{\int_0^\infty \left[\sigma^\alpha K(f,t;L_p,L_{p_1,\infty})\right]^q \frac{d\sigma}{\sigma}\right\}^{1/q}
\end{equation}
under a substitution of variable $t=\sigma^{-r}.$
We apply (\ref{inverse}) to the above equation and obtain
\begin{equation}
\|f\|_X \leq  C \left\{\int_0^\infty \left[\sigma^{-(r-\alpha)} \int_0^\sigma t^r E_t(f)_p \frac{dt}{t} \right]^q \frac{d\sigma}{\sigma}\right\}^{1/q}.
\end{equation}
By Lemma \ref{lem:Hardy}, we obtain
\[ \|f\|_X \leq C \|f\|_{A^\alpha_{p,q}}.  \]

\end{proof}

--------------------------------------------------------


\begin{thebibliography}{1}

\bibitem{bergh2012interpolation}
J{\"o}ran Bergh and Jorgen Lofstrom.
\newblock {\em Interpolation spaces: an introduction}, volume 223.
\newblock Springer Science \& Business Media, 2012.

\bibitem{devore1998nonlinear}
Ronald~A DeVore.
\newblock Nonlinear approximation.
\newblock {\em Acta numerica}, 7:51--150, 1998.

\bibitem{devore1993constructive}
Ronald~A DeVore and George~G Lorentz.
\newblock {\em Constructive approximation}, volume 303.
\newblock Springer Science \& Business Media, 1993.

\end{thebibliography}
\end{document}